\documentclass[12pt]{article}

\usepackage{latexsym,amssymb,amsmath}
\newcommand{\z}{\mathbb Z}
\newcommand{\q}{\mathbb Q}
\newcommand{\n}{\mathbb N}
\newtheorem{lem}{Lemma}[section]

\newtheorem{co}[lem]{Corollary}
\newtheorem{thm}[lem]{Theorem}
\newtheorem{prop}[lem]{Proposition}

\newenvironment{proof}{\textbf{Proof.}}{\newline\hspace*{\fill}{$\Box$}\\}

\begin{document}
\title{Non proper HNN extensions and uniform uniform exponential
growth}
\author{J.\,O.\,Button\\
Selwyn College\\
University of Cambridge\\
Cambridge CB3 9DQ\\
U.K.\\
\texttt{jb128@dpmms.cam.ac.uk}}
\date{}
\maketitle
\begin{abstract}
If a finitely generated torsion free 
group $K$ has the property that all finitely
generated subgroups $S$ of $K$ are either small or have growth constant 
bounded uniformly away from 1 
then a non proper HNN extension $G=K\rtimes_\alpha\z$
of $K$ has the same property. Here small means cyclic or, if $\alpha$ has
no periodic conjugacy classes, free abelian of bounded rank.
\end{abstract}
\section{Introduction}

If $A$ is a finite generating set for the group $G$ then the growth
function $\gamma_A:\n\mapsto\n$ of $G$ with respect to $A$ is the
number of elements in $G$ with word length at most $n$ when written
as a product of the elements of $A$ and their inverses. The exponential
growth rate $\omega(G,A)$ of $G$ with respect to $A$ is defined to be
the limit as $n$ tends to infinity of $\gamma_A(n)^{1/n}$: this limit
always exists as $\gamma_A$ is a submultiplicative function. If 
$\omega(G,A)>1$ for some $A$ then this holds for all finite generating
sets and so we have the division of finitely generated groups into those
with exponential and non-exponential word growth.

However we also have the newer concept of uniform exponential growth.
Here we define the growth constant $\omega(G)$ to be the infimum over
all finite generating sets $A$ of $\omega(G,A)$ and say that $G$ has
uniform exponential growth if this infimum is strictly bigger than 1.
(However the supremum is always infinity for groups of exponential
growth: for instance \cite{wag} Proposition 12.10(b) shows that for any
$c>1$ there is a finite set $S\subseteq G$ where the group 
$\langle S\rangle$ has $\gamma_S(n)\geq c^n$ for arbitrarily large $n$.
But now we can add a finite generating set of $G$ to $S$ without
reducing the growth function.)  
The existence of a finitely generated
group having exponential but non uniform exponential growth was eventually
established in \cite{jsw}. Nevertheless there are many classes where any
group with exponential growth is also known to have uniform exponential 
growth: here we just mention word hyperbolic groups \cite{k} and indeed
groups which are hyperbolic relative to a collection of proper subgroups
\cite{xx}, groups
which are linear over a field of characteristic 0 \cite{emo}, elementary
amenable groups \cite{osea} and 1-relator groups \cite{grdlh}.

Very recently there has been interest in the concept of what is sometimes
called uniform uniform growth: we say that a class of groups has uniform
uniform growth if there is $k>1$ such that every
group in the class has growth constant at least $k$. In particular one
consequence of the work of Breuillard and Breuillard-Gelander culminating
in \cite{breu} Corollary 1.2 
is that if $d$ is an integer then the class of non virtually soluble
linear groups of dimension $d$ over any field has uniform
uniform growth. We also have \cite{fdc} which establishes uniform uniform
growth for the fundamental groups of closed 3-manifolds which have
exponential growth. Moreover whenever we have a group $G$ with uniform
exponential 
growth there is always an opportunity to ask about uniform uniform growth:
namely is there $k>1$ such that any finitely generated
subgroup of $G$ which is not ``small'' (for instance virtually cyclic,
virtually nilpotent, virtually soluble) has growth constant at least $k$?
Note that if $G$ fails this property when small means 
virtually soluble then
$G$ cannot be a linear group by the above. For examples which do possess
this property  we have in
\cite{man} the existence of $k_S>1$ such 
that any finitely generated subgroup of the mapping class group of a
compact orientable surface $S$ is virtually abelian or has growth
constant at least $k_S$.

Another example is that in the paper \cite{grdlh}
on 1-relator groups, we have uniform uniform growth over all 1-relator
groups with $k=2^{1/4}$ apart from
cyclic groups, $\z\times\z$ or the fundamental
group of the Klein bottle. This is established using the results of 
\cite{bdlh} which obtained uniform uniform growth with the same $k$ for
most amalgamated free products. Also this $k$ applies for HNN extensions $G$ 
except where both associated subgroups are equal to the base
$H$, in which case $G$ is a semidirect product of the form $H\rtimes\z$.

In this paper our aim is to give results on when such a group $G$ has uniform
growth. In particular we consider the case when the class of
finitely generated subgroups of $H$ which are not ``small'' has uniform
uniform growth and we give conditions that will imply $G$ has uniform
growth. Moreover the methods establish that the class of
finitely generated subgroups of the new group $G$ which are 
not ``small'' also has uniform uniform growth. Sometimes the definition
of being small in $G$ is slightly wider than being small in $H$ but we
investigate when we can take both to be the same, which allows us to apply
our results to iterated HNN extensions.

In particular we show in Section 2 that the lower bound $2^{1/4}$ for the
growth constant of 1-relator groups also holds for finitely presented
groups of deficiency 1 (with exactly the same exceptions) and finitely
generated groups possessing a homomorphism onto $\z$ with infinitely
generated kernel. We then consider in Section 3 groups $G$ of the form
$K\rtimes_\alpha\z$ where $K$ is finitely generated and torsion free. We
show that if the growth constants of the non cyclic finitely generated
subgroups of $K$ are uniformly bounded away from 1 then the same is true
for the finitely generated subgroups of $G$ which are not cyclic,
$\z\times\z$ or the Klein bottle group. This has applications to the growth
constant of non proper HNN extensions of word hyperbolic groups, including
free-by-$\z$ and surface-by-$\z$ groups, as well as extensions of
Baumslag-Solitar groups. In Section 5 we prove the same result with the
non cyclic finitely generated subgroups of $K$ replaced by those finitely
generated subgroups which are not isomorphic to a free abelian group of
bounded rank. However we need to impose the condition that the
automorphism $\alpha$ forming the semidirect product has no periodic
conjugacy classes. If this condition holds then the finitely generated
subgroups of $G$ without uniform exponential growth will be the same
as those of $K$, so the result can be iterated.

In the last section we show how to adapt this theorem to allow the Klein
bottle group as an exceptional subgroup, both of $K$ and $G$ so again we
can form iterated HNN extensions with uniform exponential growth. Here
we require that $K$ is locally indicable. This has applications to the
growth of non proper HNN extensions of 
finitely generated groups of cohomological dimension (at most) 2, because
the only virtually nilpotent groups in this class are ($\{e\},\z$),
$\z\times\z$ and the Klein bottle group. The methods of proof are for the
most part standard exponential growth and group theoretical arguments,
although the Alexander polynomial also plays a role.

\section{Groups of deficiency 1}

For background in uniform exponential growth, see \cite{dlh} Chapter VI
and related references. Two
easily proved but invaluable results for a finitely generated
group $G$ are:\\
(1) If there exists a surjective homomorphism from $G$ to $Q$ then
$\omega(G)\geq \omega(Q)$.\\
(2) (Shalen-Wagreich Lemma) If $H$ is an index $i$ subgroup of $G$
then $\omega(H)\leq \omega(G)^{2i-1}$.\\
However it is not true that if $H$ is a finitely generated subgroup of
$G$ then $\omega(G)\geq\omega(H)$. For although it is true that if 
the finite set $A$
generates $G$ and $S$ is a subset of $A$ then $\omega(G,A)\geq
\omega(\langle S\rangle,S)$, so a finitely generated group with a
finitely generated subgroup of exponential growth also has exponential
growth, the example \cite{jsw} due to J.\,S.\,Wilson of the group with
non-uniform exponential growth is shown to have exponential growth
because it contains a non-abelian free group.   

We also quote here a few standard facts about HNN extensions.
We can form an HNN extension $G$ with stable letter $t$, base $H$ 
and associated subgroups $A,B$ 
whenever $H$ is a group possessing subgroups $A$ and $B$ where there exists
an isomorphism $\theta$ from $A$ to $B$. Note that if $H=A=B$ then we
obtain a semidirect product $G=H\rtimes_\theta\z$. We call
this a non-proper HNN extension whereas a proper $HNN$ extension is where
at least one of $A$ or $B$ is not equal to $H$.
Other possibilities are if one
of $A$ or $B$ is equal to $H$ in which case 
we say $G$ is an ascending HNN extension
(and if exactly one is equal we call the HNN extension strictly ascending).
Any HNN extension gives rise to a homomorphism $\chi$ onto $\z$ (called
the associated homomorphism of the HNN extension), which is defined by
sending $t$ to 1 and $H$ to 0. It is clear that if $H=A=B$ then the kernel
of $\chi$ is equal to $H$ because in this case every element of $G$ can
be written in the form $ht^i$ for $h\in H$.

Conversely a homomorphism $\chi$ from a group $G$ onto $\z$ allows us to
express $G$ as an HNN extension. However this is ambiguous if further
restrictions are not imposed: for instance 
we always have $G=\mbox{ker}(\chi)\rtimes\z$. If we insist that the base
is finitely generated then a better picture
emerges: by \cite{bns} Section 4 if we
express $G$ as two HNN extensions of this more restricted form
and the corresponding associated homomorphisms
are the same then they most both be non-proper, strictly ascending or
non-ascending together. In fact we are in the first case if and only
ker$(\chi)$ is finitely generated whereupon it is equal to the base. 
In the second case there is a small amount
of ambiguity in choosing the base but a great deal for non-ascending HNN
extensions. As an example the Baumslag Solitar group
$BS(2,3)=\langle a,t|ta^2t^{-1}=a^3\rangle$ can be formed with base $\langle
a\rangle$ equal to $\z$, but we can also write it as
$\langle a,b,t|b^2=a^3,tat^{-1}=b\rangle$ which is an HNN extension with
base the fundamental group of the trefoil knot. 
  
A finitely presented group $G$ is said to have deficiency $d$ if there
exists a presentation for $G$ consisting of $n$ generators and $n-d$
relators. It is well known that $\omega(F_n)=2n-1$ for
$F_n$ the free group of rank $n$ and Gromov conjectured that a group $G$
with deficiency $d$ has $\omega(G)\geq 2d-1$. It is readily seen using (1)
and (2) that if
$d\geq 2$ then $G$ has uniform exponential growth by quoting a famous
result of Baumslag and Pride \cite{bp} that $G$ has a finite index
subgroup surjecting onto $F_2$ but we have no uniform control over this 
index. In \cite{wlgr} J.\,S.\,Wilson established Gromov's conjecture by
use of pro-$p$ presentations. If $d=1$ then we certainly have groups $G$
for which $\omega(G)=1$, for instance $\z$, $\z\times\z$ and the Klein
bottle group $\langle a,t|tat^{-1}=a^{-1}\rangle$. But in another paper
\cite{wldf} of the same author, it was shown that any soluble
group of deficiency 1 is isomorphic to a Baumslag-Solitar group  
$G_k=\langle t,a|tat^{-1}=a^k\rangle$ with $k\in\z$ (so the above 3
groups correspond to $k=0,1,-1$ respectively). However it is known that
$G_k$ has uniform exponential growth for all other values of $k$, so it
is worth asking if non-soluble groups of deficiency 1 have uniform growth
or even uniform uniform growth.

A special class of deficiency 1 groups are those with a 2-generator
1-relator presentation which were dealt with in \cite{grdlh} (following
partial results in \cite{csgr}) where it was established that such a group
$G$ has $\omega(G)\geq 2^{1/4}$ with the above three exceptions.
This used the following theorem by de la Harpe and Bucher
in \cite{bdlh}.
\begin{thm}
A finitely generated group $G$
which is an HNN extension with base $H$ and associated subgroups $A$ and
$B$ has $\omega(G)\geq 2^{1/4}$ provided that $[H:A]+[H:B]\geq 3$.
\end{thm}
The inequality allows infinite index subgroups, and so
the condition in this Theorem is equivalent to saying that $G$ can be
expressed as a proper HNN extension. 
We can apply this in a similar way as for 1-relator group presentations.
\begin{thm}
If $G$ has a deficiency 1 presentation then $\omega(G)\geq 2^{1/4}$, with
the exception of $\z$, $\z\times\z$ and the Klein bottle group where
$\omega(G)=1$.
\end{thm}
\begin{proof}
By abelianising this presentation we see that $G$ must have a
homomorphism onto $\z$. A result \cite{bs} of Bieri and Strebel states that
any finitely presented group with a homomorphism $\chi$ 
onto $\z$ can be written
as an HNN extension with $\chi$ as the associated homomorphism
where the base $H$ and the associated subgroups 
$A,B$ are all finitely generated. If $A<H$ or $B<H$ then the above result
applies, but if $A=B=H$ then we know that $H$ is equal to ker$(\chi)$. 
Now a recent result \cite{koch} of Kochloukova is that
a finitely generated kernel of any homomorphism from a deficiency 1 group
onto $\z$ is free. If it is free of rank 0 or 1 then we obtain the three
exceptions but any free by cyclic group of the form $G=F_n\rtimes_\alpha\z$
where $n\geq 2$ has $\omega(G)\geq 3^{1/6}>2^{1/4}$. This result is
Lemma 2.3 in \cite{csgr} and we will give more details in the next section
where our aim is to generalise this proof for other groups.
\end{proof}
In \cite{grdlh} an immediate corollary of their result on 1-relator groups
is that if $G$ is finitely presented and has a finite generating set $S$
such that $\omega(G,S)<2^{1/4}$ then $G$ has deficiency at most 1.
Consequently by Theorem 2.2 we can strengthen this conclusion by saying that
$G$ has deficiency at most 0, or $G$ is one of our three exceptional
groups in which case $\omega(G,S)=1$.

It might be wondered if we can have uniform uniform exponential growth
not just for deficiency 1 groups but for any subgroup of a deficiency 1
group, or at least for any subgroup which is not virtually nilpotent.
This cannot hold because any group $G$ having a finite presentation
of deficiency at most zero can appear as a subgroup of deficiency 1:
merely add new generators to the presentation for $G$ to form the free
product $G*F_n$ until the new presentation has deficiency 1 and then $G$
will be a free factor. However the chief ingredients in the above proof
are that our group has a surjection to $\z$ and that if the kernel of this
surjection is finitely generated then this kernel belongs to a well
behaved class of groups. The first condition holds for all non-trivial
subgroups whenever $G$ is a
torsion free 1-relator group (the torsion free condition is equivalent
here to the relator not being a proper power in the free group providing
the generators for the 1-relator presentation for $G$) because it 
was proved by
Brodskii \cite{br} and Howie \cite{hw1r} that $G$ is locally indicable:
that is any (non-trivial) finitely generated subgroup of $G$ has a
homomorphism onto $\z$. As for the second condition, we will need to
invoke a longstanding conjecture on 1-relator groups (required here only
for the torsion free ones) which is that they are coherent, namely every
finitely generated subgroup is finitely presented.

\begin{prop} Suppose that $G$ is a torsion free 1-relator group which is
coherent. Then any finitely generated subgroup $H$ of $G$ has $\omega(H)
\geq 2^{1/4}$ unless $H$ is trivial, $\z$, $\z\times\z$ or the Klein bottle
group in which case $\omega(H)=1$.
\end{prop}
\begin{proof}
We know that there exists a homomorphism $\chi$ from $H$ onto $\z$ if $H$
is non-trivial. By the above $\omega(H)\geq 2^{1/4}$ if $K=\mbox{ker}(\chi)$
fails to be finitely generated. But otherwise $H=K\rtimes\z$. An old result
of Lyndon \cite{lyn} is that $G$ has cohomological dimension 2, so $H$ does
as well. The coherence assumption means that $H$ and $K$ are finitely
presented. We then apply \cite{bi} Corollary 8.6 which states that a
normal finitely presented subgroup of infinite index in a finitely
presented group of cohomological dimension 2 must be free. Thus $H$ is of
the form $F_n\rtimes\z$ and hence we obtain $\omega(H)\geq 3^{1/6}$ as before
unless $n=0$ or 1, in which case $H$ is again one of our three exceptional
groups.
\end{proof}

We have seen how the de la Harpe-Bucher result immediately gives us
uniform exponential growth for finitely presented groups having a
homomorphism to $\z$ where the kernel is infinitely generated. We finish
this section with a few words on how to extend this to finitely generated
groups with the same property. This follows as part of a wide ranging
result by Osin which is Proposition 3.2 in \cite{osea}, stating that a
finitely generated group $G$ having an infinitely generated normal subgroup
where the quotient is nilpotent of class (also called degree) $d$ has
$\omega(G)\geq 2^{1/\alpha}$, where $\alpha=3.4^{d+1}$. For our
situation where $d=1$, this gives us $\omega(G)\geq 2^{1/48}$. However we
can recover the previous $2^{1/4}$ bound.
\begin{prop}
If $G$ is a finitely generated group having a homomorphism $\chi$ onto $\z$
with infinitely generated kernel then $\omega(G)\geq 2^{1/4}$.
\end{prop}
This follows immediately from Theorem 2.1 (which does allow the base and
associated subgroups to be infinitely generated) and the following theorem,
which can be seen as a variation on the Bieri-Strebel
result, replacing the finitely presented hypothesis for $G$ with that of
being finitely generated. Note that the Bieri-Strebel result itself does
not hold for finitely generated groups as shown in \cite{bns} Section 7
by the example of $F/F''$ for $F$ a non-abelian free group.
\begin{thm}  
If $G$ is a finitely generated group with a
homomorphism onto $\z$ then either the kernel is finitely generated or
$G$ can be written as an HNN extension (with base and associated subgroups
possibly infinitely generated) where at least one of the
associated subgroups $A,B$ is not equal to the base $H$. 
\end{thm}
\begin{proof}
We use the fact that $G$ will have a presentation
(with finitely many generators but possibly infinitely many relators) where
one generator $t$ appears with zero exponent sum in each of the relators 
(and the homomorphism $\chi$ merely sends an element $g$ of $G$ to the 
exponent sum of $t$ in any word representing $g$). We will assume that there
are two other generators $x,y$ for $G$ in order to reduce numbers of 
subscripts but the idea behind the proof is exactly the same in the more
general case.

Given our presentation
\[G=\langle t,x,y| r_j(t,x,y): j\in\n\rangle\]
where $t$ has 0 exponent sum in all of the $r_j$, 
we can use Reidemeister-Schreier
rewriting to get a presentation for the kernel $K$ of $\chi$ as follows:
the generators are $x_i=t^ixt^{-i}$ and $y_i=t^iyt^{-i}$ for $i\in\z$ and
the relators $\overline{r}_{j,k}$ for $j\in\n$ and
$k\in\z$ are formed by taking the
relators $r_{j,k}=t^kr_jt^{-k}$ (which are words in $t,x,y$ with zero
exponent sum in $t$) and rewriting
them in terms of $x_i$ and $y_i$ where $i\in\z$.
We then have a (rather long) alternative presentation for $G$ of the form
\[\langle t,x_i,y_i:i\in\z|\overline r_{j,k}(x_i,y_i): j\in\n,k\in\z;
tx_it^{-1}=x_{i+1}, ty_it^{-1}=y_{i+1}:i\in\z\rangle\]

Consequently $G$ is a non-proper HNN extension with base $K$. However we 
attempt to express $G$ as a strictly ascending (and hence proper) 
extension by setting $H=A=\langle
x_i,y_i\rangle$ where now $i$ ranges over $0,1,2,\ldots$, not $\z$, and
$B=\langle x_i,y_i\rangle$ for $i=1,2,3,\ldots$. We can certainly form
the HNN extension $\Gamma=\langle H,s\rangle$ 
with base $H$, stable letter $s$ and associated subgroups $A,B$,
This is allowed as $A$ is isomorphic to $B$ inside $H$ because we 
have $t\in G$ with $tAt^{-1}=B$.
The generators for $\Gamma$ are then $s,x_i,y_i$ for
$i\geq 0$ and the relations for $\Gamma$ are the relations that hold in $H$ 
along with
\[sx_is^{-1}=x_{i+1}, sy_is^{-1}=y_{i+1}\mbox{ for }
i=0,1,2,\ldots\]  

We now alter the presentation for $\Gamma$ somewhat. We add generators
$x_i,y_i$ and relators $sx_is^{-1}=x_{i+1}$ and $sy_is^{-1}=y_{i+1}$
for negative values of $i$.
Also, rather than including all relations that hold in $H$, we take for
each original relator $r_j$ a high enough value $l(j)$ of $k$ such that the
resulting relator $\overline r_{j,l(j)}$ is written in terms of only those
$x_i,y_i$ where $i\geq 0$. Then any relation that holds amongst the
generators of $H$ also holds in $K$ and so is a consequence of the
$\overline r_{j,k}$ over all $j$ and $k$. But for each $j$ we included
$\overline r_{j,l(j)}$ in our presentation for 
$\Gamma$, and so as conjugating by
$s^{\pm 1}$ sends $\overline r_{j,k}$ to 
$\overline r_{j,k\pm 1}$, all $\overline r_{j,k}$ are
equal to the identity in $\Gamma$. Consequently in the presentation for
$\Gamma$ we can now throw in all remaining $\overline r_{j,k}$. Finally,
on renaming $s$ as $t$, we see that this presentation for $\Gamma$ is
identical to the long one for $G$.

Our hope is that $\Gamma$ is a proper HNN extension with, say, 
$x_0\notin H$. However it could be that $x_0$ and $y_0$ are both in $H$, so
can be expressed as words $v(x_i,y_i)$ and $w(x_i,y_i)$ where all $i$ 
appearing in both words are at least 1. As only finitely many $i$ can appear
in these two words, let $M$ be the maximum value occurring. Note that
$x_{-1},y_{-1}$ are in $H$ too as on conjugating by $t$
they can be expressed in terms of
$x_0,y_0,\ldots ,x_{M-1},y_{M-1}$ and then $x_0,y_0$ can be replaced.
Continuing this process, we see that $x_i,y_i\in H$ for all negative values
of $i$.

Now what we do is to run the whole proof over again but the other way,
meaning that our new $H$ and $A$ are now $\langle x_0,y_0,x_{-1},y_{-1},
\ldots\rangle$ and $B$ is set equal to $\langle x_{-1},y_{-1},x_{-2},y_{-2},
\ldots\rangle$. If we find that both $x_0$ and $y_0$ are words in $x_i,y_i$ 
for strictly negative $i$ then let $m$ be the lowest value of $i$
appearing. Just as before we see that $x_1,y_1,x_2,y_2,\ldots$ can be written
in terms of $x_m,y_m,\ldots ,x_0,y_0$. But now $K$ is finitely generated
by $x_m,y_m,\dots ,x_M,y_M$. 
\end{proof}

\section{Non-proper HNN extensions}

It was noted in \cite{butdef} that a free-by-cyclic group $G$ of the form
$F_n\rtimes\z$ for $F_n$ the free group of rank $n\geq 2$ has uniform
exponential growth by utilising the fact that the group is either large
or word hyperbolic. However a simpler and more direct argument in \cite{csgr}
Lemma 2.3 which we now examine tells us that $\omega(G)\geq 3^{1/6}$.
The previous Lemma in that paper notes that if a group $G$ is generated
by two finite sets $X$ and $Y$ then $\omega(G,X)\geq (\omega(G,Y))^{1/L}$
where $L$ is an upper bound for the length of each element of $Y$
when written as a word in $X^{\pm 1}$. 
Consequently on being given any finite set of
generators $A=\{a_1,\ldots ,a_l\}$ 
for $G=F_n\rtimes\z$, we consider the set $C$ of commutators of length 1
words, that is $C=\{[a_j^{\pm 1},a_k^{\pm 1}]:1\leq j,k\leq l\}$,
and further the set of conjugates by length 0 or 1 words of elements in $C$,
obtaining $B=\{a_i^{\pm 1}ca_i^{\mp 1}\}\cup C$ 
for $1\leq i \leq l$ and $c\in C$.
Then as we have that any element of $A\cup B$ has length at most 6 in the
elements of $A$, we obtain $\omega(G,A)^6\geq \omega(G,A\cup B)
\geq \omega(\langle B\rangle,B)$. Now $B$ has been chosen to lie in $F_n$.
Thus we are done if $\omega (\langle B\rangle)\geq 3$ 
which means we just have to show that
the free group $\langle B\rangle$ is non-abelian. 
This is straightforward but requires
using standard properties of free groups.

If we abstract this approach with the aim of applying it to other groups of
the form $G=K\rtimes\z$ where $K$ is finitely generated, we see that we
require of $K$ that there is a uniform lower bound $u>1$ for $\omega(H)$ over 
finitely generated subgroups $H$ of $K$ which are not ``small'', which here
means non-cyclic (in the above case $u=3$).
Then on being given any finite generating set $A$ for $G$ we can take words
which are of bounded length $b$ in the elements of $A$
(such as length 6 for free by
cyclic groups but this quantity has to be independent of $A$) and such that
these words lie in $K$. If these words happen to generate a subgroup of $K$
that is not ``small'' 
then we conclude as before that $\omega(G,A)\geq u^{1/b}$.

\begin{thm} Let $K$ (not equal to $\{e\}$ or $\z$)
be a finitely generated torsion-free group 
where there exists $u>1$ such that if $H$ is a finitely generated
subgroup of $K$ then either $\omega(H)\geq u$ or $H$ is cyclic. Then
for any non-proper HNN extension $G=K\rtimes\z$ we have that $G$ has
uniform exponential growth, and further $\omega(G)\geq\mbox{min }
(u^{1/6},2^{1/16})$.
\end{thm}
\begin{proof}
Given any finite generating set $A=\{a_1,\ldots ,a_l\}$ for $G$, we consider
for each $1\leq i,j,k\leq l$ the 2-generator
subgroup $H_{i,j,k}=\langle a_i,[a_j,a_k]\rangle$. Note that
if $\omega(H_{i,j,k})=c$ then $\omega(G,A)\geq c^{1/4}$. If
$a_i\in K$ then $H_{i,j,k}\leq K$ so the inequality
would be true with $c=u$ unless
$H_{i,j,k}$ is cyclic in which case $a_i$ commutes with $[a_j,a_k]$.
But if $a_i\notin K$ then the associated homomorphism from $G$ to $\z$ with
kernel $K$ restricts to a non-trivial (so without loss of generality
surjective) homomorphism $\chi:H_{i,j,k}\rightarrow \z$. If the kernel
$L=H_{i,j,k}\cap K$ 
of this restriction is infinitely generated then we immediately get
by Proposition 2.4 that $\omega(H_{i,j,k})\geq 2^{1/4}$ and so $\omega(G)\geq
2^{1/16}$. But on setting $t=a_i$ and $x=[a_j,a_k]$ 
we have that $L$ is the
normal closure of $x$ in $H_{i,j,k}$ and so is generated by $x_n=t^nxt^{-n}$
for $n\in\z$. 

However we can now assume that $L$ is finitely generated. Let us define
$L_0$ to be the cyclic group $\langle x_0\rangle$ 
and $L_{\pm 1}$ to be the 2-generator group
$\langle x_0,x_{\pm 1}\rangle$ for each choice of sign.
If one of $L_{\pm 1}$
is not cyclic then we have $\omega(L_{\pm 1})\geq u$ and so as the length of
$x_0$ and $x_{\pm 1}$ is at most 6 in $A$, we obtain $\omega(G,A)\geq u^{1/6}$.

Thus we are done unless we find for all $i,j,k$
that $H_{i,j,k}$ is cyclic whenever $a_i$ is in $K$ and 
both of $L_{\pm 1}$ are cyclic whenever $a_i$ is not in $K$. Let
us assume for now that this always implies that $L_{\pm 1}=L_0$ and explain
how we finish. Having $L_1=L_0$ means that $t$ conjugates $x$ into a power
of itself, and $L_{-1}=L_0$ gives us the same for $t^{-1}$. This means that
$t\langle x\rangle t^{-1}=\langle x\rangle$ and furthermore that
$txt^{-1}=x^{\pm 1}$. If we now fix $j$ and $k$ but let $i$ vary, we conclude
that $\langle [a_j,a_k]\rangle$ is conjugated to itself by every element in
the generating set $A$ and so $\langle [a_j,a_k]\rangle$ is a normal cyclic
subgroup of $G$. Let $P$ be the product over $j$ and $k$
of these subgroups. 
By Fitting's Theorem, a (finite) product of normal nilpotent 
subgroups of $G$ is also nilpotent, and as $P$ lies in $K$ we have that
$P$ is cyclic. But $G/P$ is abelian as now all generators commute.
Moreover if $p$ is a generator for $P$ then we must have
$gpg^{-1}=p^{\pm 1}$ for all $g\in G$ because $P$ is normal in $G$. This
means that the centraliser $C(p)$ of $p$ in $G$
contains all squares, so contains the subgroup $G^2$ generated by all squares.
Now $G/G^2$ is a group of exponent 2, hence abelian, and is also finitely
generated so must be finite. But $P$ is contained in $G^2$ and hence is in the
centre of $G^2$. Quotienting $G^2$ by $P$, we are left with the abelian
group $G^2/P$, telling us that $G^2$ is nilpotent and $G$ is virtually 
nilpotent. This means that $K$ is too, so $\omega(K)=1$ which leaves only
$K=\{e\}$ or $\z$.

However we must now address the fact that $L_1$ being cyclic does not
mean that $L_0=L_1$, only that $L_0$ is a cyclic subgroup of $L_1$ and so
is of finite index. Consequently there is a power of $x_1$ which lies in $L_0$
and so we must have a relation holding in $L$ of the form $x_0^\alpha 
x_1^\beta$. We can assume that $\alpha$ and $\beta$ are coprime because
$x_0$ and $x_1$ commute and $L$ is torsion free. Also if $\beta=\pm 1$ then 
we have $L_1=L_0$ anyway.

We now set $L_2=
\langle x_0,x_1,x_2\rangle$ and so on. We have that either $L_2$ is non-cyclic,
so has uniformly exponential growth, or it remains cyclic
and contains $L_1$ with
finite index. But as $L$ is finitely generated, this sequence of subgroups
must eventually stabilise when $L_n=L_{n+1}=L$. We need to deal with two
possibilities. The first is that $L_n$ is still cyclic. The other is that
at some point $m$ (possibly for a large $m$ but which could be a lot 
smaller than $n$) 
we have that $L_m$ stops being cyclic and so has uniformly exponential
growth. We will show that in fact neither of these cases can occur. The
first uses the Alexander polynomial of $H_{i,j,k}$ and is based on the 
fact that the degree of the polynomial is $\beta_1(L;\q)$ but this
contradicts $L$ being finitely generated. The second case is again based
on $L$ being finitely generated to get that $L_n$ is still abelian.

We briefly review the facts we will need about the Alexander polynomial.
The Alexander polynomial $\Delta(t)$ is defined for a homomorphism
$\chi$ from a finitely generated group $G$ onto $\z$. Usually it is
supposed that $G$ is finitely presented but we will only ever need it
for 2-generator groups $G=\langle t,x\rangle$ where we can suppose that
$\chi(t)=1$ and $\chi(x)=0$. In this restricted setting, the results we
require will continue to apply even if $G$ is not finitely presented.
Let us set $K=\mbox{ker}(\chi)$ and note that $K$ is generated by
$x_i=t^ixt^{-i}$ for $i\in\z$. On abelianising we have that $K/K'$ is
a finitely generated module $M$ over the unique factorisation domain
$R=\z[t^{\pm 1}]$ where the polynomial variable $t$ acts on elements of
$M$ via conjugation by the group element $t$. In our 2-generator case 
we have that $M$ is a cyclic module $R/I$ generated by $x_0$
and here we define the 
Alexander polynomial $\Delta(t)$ to be the highest common factor of the
elements in $I$, or equivalently the generator of the smallest principal
ideal containing $I$. It is defined up to multiplication by units which
are $\pm t^k$ for $k\in\z$. 

The three facts that we require here are:\\
(1) The degree of $\Delta(t)$ (meaning the degree of the 
highest power of $t$ minus the lowest) is equal to the first Betti number 
$\beta_1(K;\q)$ which is defined to be the dimension of the vector space
$H_1(K;\q)=H_1(K;\z)\otimes_\z\q$ over $\q$. This is clear here because
$H_1(K;\z)=K/K'$ so we just work over the ring $\q[t^{\pm 1}]$ instead.\\
(2) If $K$ is finitely generated as a group then $\Delta(t)$ is monic at
both ends: that is the highest and lowest coefficients must both be $\pm 1$.
This follows because the sequence of subgroups $S_0=\langle x_0\rangle$,
$S_1=\langle x_0,x_1\rangle$, $S_2=\langle x_0,x_1,x_2\rangle,\ldots $ must
stabilise at $n$ say, so that $x_n$ is a word in $x_0,\ldots ,x_{n-1}$.
On abelianising, this gives us that $t^n+a_{n-1}t^{n-1}+\ldots +a_0$ must
hold in $M$, thus $\Delta(t)$ divides a monic polynomial so is itself
monic. As for the lowest coefficient, we argue the other way with
$x_0,x_{-1},x_{-2},\ldots$.\\
(3) If we consider the sequence of subgroups $S_i$ above then we must
clearly have $\beta_1(S_i)\leq i+1$ as this is bounded by the number of
generators of a group. But on taking the first $i$ where 
$\beta_1(S_i)\leq i$ (if one exists) we have $\beta_1(S_j)$ is non
increasing for
$j\geq i$. This is because once we have a relationship of the form
$\alpha_0 x_0+\ldots +\alpha_i x_i=0$ with $\alpha_i\neq 0$ in the
abelianisation of $S_i$, this relation will also hold in the abelianisation
of $S_{i+1}$ along with the same relation with the subscripts shifted up
by 1. Thus in $S_{i+1}/S_{i+1}'$ we see that $x_{i+1}$ will lie in the span
of $x_0,\ldots ,x_i$ which means we cannot increase $\beta_1$. We can now
repeat this argument. 

We apply the above to the homomorphism $\chi:H_{i,j,k}\rightarrow\z$ 
with kernel $L$. 
As the relation $x_0^\alpha x_1^\beta$ holds in $L$, we must have that
the Alexander polynomial $\Delta(t)$ of $\chi$ divides $\beta t+\alpha$. 
However $L$ being finitely generated means that $\Delta(t)$ must
be monic, so equal to
$\pm 1$ because $\beta\neq\pm 1$ and $\alpha$ and $\beta$
are coprime. But this implies that $\beta_1(L;\q)=0$ which contradicts
the fact that $L\cong\z$.

As for when $L$ is non-cyclic, we have
$L=\langle x_0,x_1,\ldots ,x_{n-1},x_n\rangle$. By using the conjugation
action of $t$ on $L$, we have that $x_i^\alpha x_{i+1}^\beta=e$ for all
$i\in\z$ and that $x_i$ commutes with $x_{i\pm 1}$.
We will show that a very high power of $x_0$ is in the centre of $L$.
Note that $x_0^\alpha$ is a power of $x_1$ and so commutes with
$x_2$. Moreover we have $x_0^{\alpha^2}=(x_1^\alpha)^{-\beta}$ which
commutes with $x_3$ because $x_1^\alpha$ does. Continuing in this
way, we see that $x_0^{\alpha^{n-1}}$ commutes with $x_n$ and all $x_i$
from $0$ to $n$, thus is in the centre $Z(L)$ of $L$. But $Z(L)$ is
characteristic in $L$, hence invariant under conjugation by $t$ so
$x_i^{\alpha^{n-1}}$ is in $Z(L)$ for all $i$.

To complete this argument we now start with $x_n$ and work backwards,
giving us that $x_i^{\beta^{n-1}}$ is in $Z(L)$ too. But $\alpha^{n-1}$
and $\beta^{n-1}$ are coprime, so $x_i\in Z(L)$ for all $i$ thus $L$ is
abelian. 

Thus in both cases we have shown that $L_1$ being cyclic implies that
$L_0=L_1$. We now repeat the argument with $L_{-1}$ in place of $L_1$.
\end{proof}

Our hypothesis required that every finitely generated subgroup of $K$ was
either ``small'' or had uniform exponential growth bounded below away from
1. However a straightforward adaptation of our argument obtains the same
conclusion for $G$, although with a slight widening of the concept of
``small''.

\begin{co} Let $K$ be a finitely generated torsion-free group 
where there exists $u>1$ such that if $H$ is a finitely generated
subgroup of $K$ then either $\omega(H)\geq u$ or $H$ is cyclic. 
Let $G=K\rtimes\z$ be any non-proper HNN extension of $K$. Then for
any finitely generated subgroup $S$ of $G$ we have that $S$ has
uniform exponential growth with $\omega(S)\geq\mbox{min }
(u^{1/6},2^{1/16})$ or
$S$ is $\{e\}$,$\z$, $\z\times\z$ or the Klein bottle group. 
\end{co}
\begin{proof}
We follow the proof of Theorem 3.1 with $G$ replaced by $S$. If $S$ lies
in $K$ anyway then we have our bound for $\omega(S)$ unless $S=\{e\}$ or
$\z$. If $S$ has elements outside $K$ then the natural homomorphism
from $G$ to $\z$ with kernel $K$
restricts to a non-trivial homomorphism from $S$. If
the kernel $R=S\cap K$ of the restriction is infinitely
generated then we have $\omega(S)\geq 2^{1/4}$ by Proposition 2.4. Otherwise
we can regard $S$ as a semidirect product $R\rtimes\z$ and so
can replace $G$ with $S$ and $K$ with $R$ in Theorem 3.1. The hypotheses
are satisfied unless $R=\{e\}$ or $\z$, in which case $S$ can only be
$\z$, $\z\times\z$ or the Klein bottle group. 
\end{proof}

\section{Applications and examples}

The first examples of groups we can think of where every subgroup is
either cyclic or ``big'' are of course free groups. This is not
surprising in light of the fact that Theorem 3.1 takes as its starting
point the Ceccherini-Silberstein and Grigorchuk result that 
$G=F_n\rtimes\z$ has $\omega(G)\geq 3^{1/6}$ for $F_n$ the finitely
generated free group of rank $n\geq 2$. Here we can obtain a slight
generalisation.
\begin{prop}
If $G=F\rtimes\z$ where $F$ is a free group of any cardinality then
all finitely generated subgroups $H$ of $G$ have $\omega(H)\geq 3^{1/6}$
unless $H$ is $\{e\}$, $\z$, $\z\times\z$ or the Klein bottle group.
\end{prop}
\begin{proof}
Either $H$ is contained in $F$, in which case $H$ is free and we have
$\omega(H)\geq 3$ or $H=\{e\}$ or $\z$, or $H$ is a semidirect product
$(H\cap F)\rtimes\z$. If $H\cap F$ is finitely generated then we have
$\omega(H)\geq 3^{1/6}$ or $H=\z,\z\times\z$ or the Klein bottle group.
If $H\cap F$ is infinitely generated then, although we know that 
$\omega(H)\geq 2^{1/4}$, we can do better by using an observation of
\cite{butdef} which is based on the results of \cite{fh}. The
latter proves that $H$ is finitely presented and it is pointed out in
Section 5 of the former that $H$ must have a presentation with deficiency
at least 2, so we have $\omega(H)\geq 3$ by J.\,S.\,Wilson's result.
\end{proof}

Our next examples are the surface groups $\Gamma_g$; these are the
fundamental groups of the closed orientable surface of genus $g\geq 2$.
It is known that $\omega(\Gamma_g)\geq 4g-3$ (see \cite{dlh} VII.B.15
although it is not known if this is best possible). Using Theorem 3.1
means we can treat HNN extensions of free groups and surface groups in
a unified fashion.

\begin{co} If $G=\Gamma_g\rtimes\z$ then we have $\omega(S)\geq 3^{1/6}$
for any finitely generated subgroup $S$ of $G$, apart from
$S$ equal to $\{e\}$, $\z$, $\z\times\z$ or the Klein bottle group.
\end{co}
\begin{proof} It is well known that the subgroups of $\Gamma_g$ are 
isomorphic to $\Gamma_h$ in the finite index case and are free if 
they are of infinite index. Therefore the conditions of Theorem 3.1
and Corollary 3.2 apply with $u=3$. By the proof of Theorem 3.1,
we either have that one of the subgroups $H_{i,j,k}$ of $S$ lies in $\Gamma_g$
and is non-cyclic, in which case we have $\omega(S)\geq
u^{1/4}$, or for some $i,j,k$ we have that the kernel
$L$ is finitely generated and
one of $L_{\pm 1}$ is non-cyclic so $\omega(S)\geq u^{1/6}$, or
$L$ is infinitely generated. If that occurs then $H_{i,j,k}$ is 
(free of infinite rank)-by-$\z$ and so $\omega(H_{i,j,k})\geq 3$ by 
Proposition 4.1, giving $\omega(S)\geq 3^{1/4}$.
\end{proof}

The groups $G$ in Corollary 4.2 are all
fundamental groups of closed 3-manifolds,
where the 3-manifold is fibred over the circle. In \cite{fdc} it is shown
(using Geometrisation) that we have uniform uniform growth for the class
of fundamental groups of closed 3-manifolds, with the exception of the
virtually nilpotent groups.

A much wider class of groups (including free and surface groups) 
where every subgroup
is either cyclic or contains a non-abelian free group is the class of
torsion free word hyperbolic groups (note that a torsion free virtually cyclic
group must be $\z$ or $\{e\}$, for instance by \cite{hemp} Lemma 11.4).
Therefore Theorem 3.1 and Corollary 3.2 will apply to such a group $G$
whenever we have a lower bound away from 1 of the growth constant
of all non-cyclic
finitely generated subgroups. We present two cases where this would be so.
\begin{co} Let $K$ be a torsion free word hyperbolic linear group. Then there
exists a constant $k>1$ such that if 
$G=K\rtimes\z$ is any non-proper HNN extension of $K$ then we have 
$\omega(S)\geq k$ for any finitely generated subgroup $S$ of $G$ that is
not isomorphic to $\{e\}$, $\z$, $\z\times\z$ or the Klein bottle group.
\end{co} 
\begin{proof} By the results \cite{breu}, \cite{brgl}
of Breuillard and Breuillard-Gelander, we
have that for any integer $d>1$ there is a constant $c(d)>1$ such that if
K is a finitely generated non-virtually solvable subgroup of 
$GL_d(\mathbb F)$, where $\mathbb F$ is any field, then $\omega(K)\geq c(d)$.
But if $K$ is also torsion free and word hyperbolic then any non-cyclic
finitely generated subgroup $H$ (though it may not be word hyperbolic) will 
also have $\omega(H)\geq c(d)$. Therefore Corollary 3.2 applies with
$u=c(d)$.
\end{proof}
Note that although $G$ above will be torsion free, it will not in general
be word hyperbolic because it could easily contain $\z\times\z$, as
discussed in the next section. The question
of whether $G$ is linear seems interesting; it is open even in the case
where $K$ is the free group $F_r$ for $r\geq 3$.

Another result that could be of use here is that of Arzhantseva and Lysenok
in \cite{arzl}. This states that if $K$ is a word hyperbolic group then
there is an $\alpha>0$, effectively computable from $K$, such that for
any finitely generated subgroup $H$ of $K$ which is not virtually cyclic
and any finite generating set $C$ for $H$, we have 
$\omega(H,C)\geq\alpha|C|$. This implies that if $K$ is torsion free word
hyperbolic and the value of $\alpha$ above is greater than $1/2$ then
Corollary 3.2 applies to $K$ because when $H$ is non-cyclic we have 
$|C|\geq 2$ and so $\omega(H,C)\geq u=2\alpha >1$.

We should say that if a torsion free 1-ended word hyperbolic group $H$ has
trivial JSJ decomposition then Out$(H)$ is finite and therefore any group
of the form $H\rtimes_\alpha \z$ has uniform exponential growth. This is
because if $\alpha$ has order $a$ in Out$(H)$ then the degree $a$ 
cyclic cover of $H\rtimes_\alpha\z$ is isomorphic to $H\times\z$ and so
$\omega(H\rtimes\z)^{2a-1}\geq\omega(H\times\z)\geq\omega(H)>1$.
However $a$ could be arbitrarily high and moreover if Out($H$) is infinite 
then this argument will only apply to the automorphisms of finite order.

Finally we wish to display some examples which are not word hyperbolic;
indeed which are as far away from being word hyperbolic as possible.
For non-zero integers $m,n$ let $BS(m,n)$ be the Baumslag-Solitar group
with presentation $\langle a,t|ta^mt^{-1}=a^n\rangle$. These are
``poison'' subgroups in the sense that any group containing one of these
cannot be word hyperbolic. However $BS(1,n)$ does not
contain $\z\times\z$ if $n\neq\pm 1$. Thus we obtain:
\begin{co}
Let $G$ be any non-proper HNN extension of $B=BS(1,n)$ for $n\neq\pm 1$ then
we have $\omega(S)\geq 2^{1/24}$ for 
any finitely generated subgroup $S$ of $G$ that is
not isomorphic to $\{e\}$, $\z$, $\z\times\z$ or the Klein bottle group.
\end{co}
\begin{proof}
The group $B$ is a strictly ascending HNN extension with base $\z$ so the
associated homomorphism has a kernel which is not finitely generated but
which is an ascending union of copies of $\z$, so is locally $\z$. 
If $H$ is a finitely generated subgroup of $B$ which is not in this
kernel, then if it intersects the kernel in an infinitely generated subgroup
we have $\omega(H)\geq 2^{1/4}$. If however the intersection is finitely
generated then it must be $\{e\}$ or $\z$, so $H$ is $\z$, $\z\times\z$
or the Klein bottle group but we have ruled out the last two. Therefore
Corollary 3.2 applies with $u=2^{1/4}$.
\end{proof}

The above discussion also applies to generalised Baumslag-Solitar groups.
These are finitely generated groups which act on a tree with all edge
and vertex stabilisers infinite cyclic. If such a group does not contain
$\z\times\z$ then Corollary 4.4 will apply. To see this we note that any
generalised Baumslag-Solitar group has a surjection to $\z$ and any
finitely generated subgroup is either free or also a generalised
Baumslag-Solitar group. If the kernel of this surjection is infinitely
generated then we have the $2^{1/4}$ bound but if the kernel is finitely
generated then we use the fact that generalised Baumslag-Solitar groups
have cohomological dimension 2 and are coherent. By the result of Bieri
mentioned before in Proposition 2.3, we have that this kernel must be
free and hence we get in this case a lower bound of $3^{1/6}$ once we
have ruled out $\z\times\z$ and the Klein bottle.

The paper \cite{lev} describes the automorphism group of a
generalised Baumslag Solitar group, which in some cases can be a big
group. It also quotes the results used above and provides references.

\section{Periodic conjugacy classes}

We know that a group containing $\z\times\z$ cannot be word hyperbolic and
so in general $G=K\rtimes_\alpha\z$ will not be word hyperbolic even if $K$
is. We can see this because if the automorphism $\alpha$ has a periodic
conjugacy class, that is there exists a non-identity element $k$ of $K$
and a non-zero integer $n$ such that $\alpha^n(k)$ is conjugate by $c\in K$
to $k$, then $t^nkt^{-n}=ckc^{-1}$ where conjugation by $t$ 
acts as $\alpha$ on $K$.
Thus $c^{-1}t^n$ commutes with $k$ and if $K$ is torsion free then this
must generate $\z\times\z$ because the image in $\z$ of any power of 
$c^{-1}t^n$ is never trivial so no power can lie in $K$. In fact the
converse is true too.

\begin{prop} Let $K$ be a group that does not contain $\z\times\z$ and
let $\alpha$ be an automorphism of $K$. Then if $G=K\rtimes_\alpha\z$ 
contains
$\z\times\z$, we have that $\alpha$ has a periodic conjugacy class.
\end{prop}
\begin{proof}
We take two generators of $\z\times\z$ in $G$ which we can write as
$x=kt^i$ and $y=lt^j$ for $k,l\in K$. Without loss of generality $i\neq 0$
because if both are then $x,y\in K$. We then set $z=x^jy^{-i}$ which (on
taking the exponent sum of $t$ in $z$) will lie in $K$ and will not equal
the identity. But then $xzx^{-1}=z$ implies that $\alpha^i(z)=k^{-1}zk$.
\end{proof}

However it is far from immediate that an absence of $\z\times\z$ in $G$
along with $K$ being word hyperbolic is enough to imply that $G$ is too. 
Even when
$K$ is the free group $F_n$ one requires the combined results of
\cite{bf}, \cite{bfad} and \cite{brink} to prove this. Even so, as word
hyperbolicity is not explicitly mentioned in Theorem 3.1 or Corollary 3.2,
we can use the absence of periodic conjugacy classes to obtain iterated
versions of these results.
\begin{co} Let $K$ be a finitely generated torsion-free group 
where there exists $u>1$ such that if $H$ is a finitely generated
subgroup of $K$ then either $\omega(H)\geq u$ or $H$ is cyclic. 
Let $\alpha$ be an automorphism of $K$ without periodic conjugacy classes
and let $G=K\rtimes_\alpha\z$.
Now let $\beta$ be any automorphism of $G$ and let us form the repeated 
non-proper HNN extension 
$D=(K\rtimes_\alpha\z)\rtimes_\beta\z$.
Then there exists $w>1$ such that 
for any finitely generated subgroup $S$ of $D$ we have that $S$ has
uniform exponential growth with $\omega(S)\geq w$ or
$S$ is $\{e\}$,$\z$, $\z\times\z$ or the Klein bottle group.
\end{co}
\begin{proof}
On applying Corollary 3.2 to $K$, we obtain that $G$ is
finitely generated and torsion free, with $v>1$ such that any finitely
generated subgroup $H$ of $G$ having $\omega(H)<v$ must be cyclic or
isomorphic to $\z\times\z$ or the Klein bottle. But the last two cases
are ruled out by Proposition 5.1 if $\alpha$ has no periodic conjugacy
classes. Now we can apply Corollary 3.2 again but this time to $G$.
\end{proof}
In particular, by the above and Proposition 4.1,  
if $G=F_n\rtimes\z$ is a word hyperbolic free-by-cyclic group
then any finitely generated
subgroup $D$ of a group of the form $G\rtimes\z$ has
$\omega(D)\geq 3^{1/36}$. The same is true by Corollary 4.2
if $G$ is the fundamental group
of a closed hyperbolic 3-manifold which is fibred over the circle.
Also if we have any number of repeated non-proper
HNN extensions of a group $K$ satisfying the conditions of Theorem 3.1
where all automorphisms (except possibly the last) have no periodic
conjugacy classes then the resulting group has uniform exponential growth
with constant depending only on $K$ and the number of HNN extensions.

We might hope to develop versions of Theorem 3.1 and Corollary 3.2 where
the allowable ``small groups'' are more than just cyclic, for instance we
could permit free abelian groups of bounded rank. The main problem here
appears to be finding an ``exit strategy'' in the sense that we may
always find that $L$ in the proof turns out to be abelian but it is
hard to see what this implies for the whole group $G$. Nevertheless the
absence of periodic conjugacy classes allows us to get round this.

\begin{thm} Let $K$ (not equal to $\{e\}$)
be a finitely generated torsion-free group 
where there exists $u>1$ and $d\in\n$ such that if $H$ is a 
finitely generated 
subgroup of $K$ then either $\omega(H)\geq u$ or $H\cong\z^n$ for $n\leq d$.
Then there exists $k>1$ depending only on $u$ and $d$ such that
for any non-proper HNN extension $G=K\rtimes_\alpha \z$ where $\alpha$
has no periodic conjugacy classes, we have $\omega(G)\geq k$.
\end{thm}
\begin{proof}     
We follow the proof of Theorem 3.1 and indicate where we need to strengthen
it. Suppose we have $H_{i,j,k}$ not contained in $K$ with finitely
generated kernel $L\neq\{e\}$. 
(Note that there will be some generator $a_i$ of $G$ which
is not in $K$, and some pair of generators $a_j,a_k$ which do not commute
or else $G$ is abelian and so has periodic conjugacy classes.
Thus this case will occur for any generating set.)
One considers as before the ascending sequence 
of subgroups 
$L_0\cong\z,L_1,L_2,\ldots$. Suppose that $L_{s-1}$ is (free)
abelian then we must have on moving from $L_{s-1}$ to $L_s$ that either
$L_s$ is free abelian of one rank higher, or $L_s$ and all further
$L_{s+1},L_{s+2},\ldots$ are no longer abelian, or we ``stick'', which is
where $L_s$ remains abelian but has the same rank as $L_{s-1}$. 

Suppose we never stick but drop out of the abelian category then we must
have $L_s$ non-abelian for $s\leq d$, in which case we have $\omega(L_s)
\geq u$ for $L_s=\langle x,txt^{-1},\ldots ,t^sxt^{-s}\rangle$ with each
element of this generating set having length at most $2d+4$ in $A$, so
$\omega(G,A)\geq u^{1/(2d+4)}$.

Now let us see what happens if we do stick, supposing that $s$ is the
first time this has occurred so that $L_{s-1}\cong L_s\cong\z^s$ with 
$s\leq d$. If $L_{s-1}=L_s$ then we have $L_{s+1}=L_{s+2}=\ldots$. We now
repeat our argument in the other direction, starting with
$M_0=\langle x_{s-1},x_{s-2},\ldots ,x_0\rangle$ and 
$M_1=\langle M_0,x_{-1}\rangle$,
$M_2=\langle M_1,x_{-2}\rangle$ and so on.
Once again we either increase the rank by 1 but
can only do this a maximum of $d-s$ times, or we are no longer abelian
but then the above bound again applies for $\omega(G,A)$, or we again
stick. If further we have $M_{t-1}=M_t$ in our new sequence of subgroups
then we have found that $L$ is isomorphic to $\z^n$ for $n\leq d$ and 
consequently the subgroup $H_{i,j,k}=\langle t,x\rangle$ is 
$\z^n\rtimes_\beta\z$ where conjugation by $t$ induces the automorphism
$\beta$ of $\z^n$.

Now a lot is known about the word growth of groups of this form. Let $M$ be
the corresponding element of $GL(n,\z)$ obtained from $\beta$. Set $m$ to
be the modulus of the largest eigenvalue of $M$; by a famous result of
Kronecker we have that $m>1$ unless all solutions of the characteristic
equation are roots of unity. In the former case our group $H_{i,j,k}$
has uniform exponential growth with growth constant that depends only
on $m$ and this again translates into a lower bound for $\omega(G,A)$.
Moreover as $m$ varies, this growth constant is bounded away from 1 
if $m$ is; see for instance the example in \cite{bdlh}. 
Whether or not our growth
constants are bounded away from 1 over all $m$ would be implied by a
positive answer to a longstanding question of Lehmer on Mahler measure. 
However we are in a position where $M$ is an $n$ by $n$ matrix for $n\leq d$.
When the degree of the polynomial is bounded we can use \cite{blamnt} which
states that if $p$ is a monic polynomial in $\z[t]$ of degree at most $d$
with $m\leq 1+1/(30d^2\mbox{ log }6d)$ 
then the solutions of $p$ are all roots of
unity, where $m$ is the modulus of the largest root.
Consequently we have uniform exponential growth for $H_{i,j,k}$
with growth constant that depends only on $d$. 

In the case where all eigenvalues are roots of unity our group
is virtually nilpotent. Let one of these eigenvalues have order $r$ then
the subgroup $\langle t^r,x\rangle$ 
of $H_{i,j,k}$ corresponds to the matrix $M^r$ which
will have 1 as an eigenvalue. Consequently there will exist a non-zero
$v\in\q^n$, and
also in $\z^n$ by clearing out denominators, with $M^rv=v$. We can think of
this $v\in\z^n$ as an element of $H_{i,j,k}\cap K$, for $K$ the kernel of
the natural homomorphism from $G$ to $\z$. 
If $s\in G$ induces the automorphism
$\alpha$ of $K$ by conjugation then we have $t=ks^i$ for $k\in K$ and
$i\neq 0$, obtaining $t^r=ls^{ir}$ for $l$ also in $K$. 
Thus we have $v=t^rvt^{-r}=ls^{ir}vs^{-ir}l^{-1}$ so 
$\alpha^{ir}(v)=l^{-1}vl$, giving us a periodic conjugacy class for $\alpha$. 

We now have to consider what happens when we stick for the first time at
$s$, so that $L_{s-1}\cong L_s\cong\z^s$, but where $L_{s-1}$ is not
equal to $L_s$. We will have $L_{s-1}$ of finite index in $L_s$ and a
relation holding of the form $x_0^{\alpha_0}\ldots x_s^{\alpha_s}=e$ where
(as $L_s$ is free abelian) there is no common factor of the integers
$\alpha_0,\ldots ,\alpha_s$ and (as $L_{s-1}<L_s$) $\alpha_s\neq\pm 1$.
As this relation continues to hold by conjugation amongst
$x_1,\ldots ,x_{s+1}$ and so on, we must have that either $L_{s+1}$ is
no longer abelian or it is abelian with $L_s\leq_f L_{s+1}$.

In the case where we never reach a non-abelian subgroup, the previous
proof generalises in a straightforward manner. We will have an ascending
chain of finite index subgroups until we terminate at $L_t=L\cong\z^s$.
Now we invoke the Alexander polynomial argument again which is that it must
be a monic polynomial (as $L$ is finitely generated) of degree
$\beta_1(L;\q)=s$. But it must also divide 
$\alpha_0+\alpha_1t+\ldots +\alpha_st^s$ which is a contradiction.

However the case where along the way we stop being abelian is
more awkward. We suppose that $L=\langle x_0,x_1,\ldots ,x_M\rangle$ for
some large $M$ and is non-abelian, with 
\begin{equation}
x_0^{\alpha_0}\ldots x_s^{\alpha_s}=e 
\end{equation}
holding as before. Our previous
argument immediately generalises to showing that $x_i^b$
is in the centre $Z(L)$ for $b$ a high power of $\alpha_0$, and
similarly for a high power of
$\alpha_s$. If $\alpha_0$ and $\alpha_s$ are coprime
then we are immediately done but this need not be so, thus we have to proceed
more carefully.

Let $c$ be the highest common factor of $\alpha_0$ and $\alpha_s$. Using
the old argument, we know
that there is some integer $a$ such that for all $0\leq i\leq M$, 
we have $x_i$ raised to the power of $c^a$
is in the centre $Z(L)$. Let us now work modulo the
centre and take a prime $p$ dividing $c$, so that there exists $r$ coprime
to $p$ and a potentially high integer $k$ with $x_i^{rp^k}=e$. We then
raise the equation (1) to the power of $rp^{k-1}$ so that if $j_1,
\ldots ,j_l$ are the values of $i$ where $p$ does not divide $\alpha_i$,
we have 
\[x_{j_1}^{\alpha_{j_1}rp^{k-1}}\ldots x_{j_l}^{\alpha_{j_l}rp^{k-1}}=e.\]
Consequently $x_{j_l}^{\alpha_{j_l}rp^{k-1}}$ is in the subgroup
generated by $x_i^{rp^{k-1}}$ for values of $i$ lower than $j_l$. But
$\alpha_{j_l}$ is coprime to $p$ so in fact $x_{j_l}^{rp^{k-1}}$ is in this
subgroup too. But now on conjugating this relation so that $j_l=M$ and
working down, we obtain that $x_M^{rp^{k-1}}$ is in the centre of $L/Z(L)$.
We can now quotient out by the new centre, obtaining $x_i^{rp^{k-1}}=e$
and then repeat the argument with $k$ lowered by 1. Continuing
in this way we reduce $k$ until after successive quotients we have
$x_i^r=e$. We can now replace $c$ with $r$, pull out another prime from
$r$ and repeat the argument, always quotienting out by the centre.
Eventually we have $x_i=e$ so that $L$ is actually nilpotent. 
As it is a subgroup of $K$, we conclude that $L$ was abelian all along. 
\end{proof}

Again we have our usual corollary.
\begin{co}
Let $K$ be a finitely generated torsion-free group 
where there exists $u>1$ and $d\in\n$ such that if $H$ is a 
finitely generated 
subgroup of $K$ then either $\omega(H)\geq u$ or $H\cong\z^n$ for $n\leq d$.
Then there exists $k>1$ depending only on $u$ and $d$ such that
for any non-proper HNN extension $G=K\rtimes_\alpha \z$ where $\alpha$
has no periodic conjugacy classes, and any finitely generated subgroup
$S$ of $G$, we have $\omega(S)\geq k$ or $S\cong\z^n$.
\end{co}
\begin{proof}
We are done if $S\leq K$ or if $S\cap K$ is infinitely generated. If
$S\cap K=R$ is finitely generated then we have $S=R\rtimes_\beta\z$
where the action of $\beta$ is conjugation by $ks^i$ for $k\in K$ and
where conjugation by $s$ on $K$ induces $\alpha$. We can then apply
Theorem 5.3 to $S$ unless $\beta$ has periodic conjugacy classes, say
$\beta^j(r)=lrl^{-1}$ for $r,l\in R$. But as $(ks^i)^j=k_0s^{ij}$ for
$k_0\in K$, we have $s^{ij}rs^{-ij}=k_0^{-1}lrl^{-1}k_0$ so $\alpha$ has
periodic conjugacy classes too.
\end{proof}

\section{The Klein bottle group}

It would be good to extend Theorem 5.3 to cases where other small groups
are allowed. In this section we will show how to alter the argument to
incorporate the Klein bottle group given by the presentation
$\langle\alpha,\beta|\beta\alpha\beta^{-1}=\alpha^{-1}\rangle$
because we have often seen results where the groups left over are cyclic,
$\z\times\z$ or the Klein bottle. However the ad hoc nature of this case
means that we need to impose an extra condition on our group $K$, which is
that it is locally indicable.
\begin{thm} Let $K$ (not equal to $\{e\}$)
be a finitely generated locally indicable group 
where there exists $u>1$ and $d\in\n$ such that if $H$ is a 
finitely generated 
subgroup of $K$ then either $\omega(H)\geq u$ or $H\cong\z^n$ for $n\leq d$
or $H$ is isomorphic to the Klein bottle group.
Then there exists $k>1$ depending only on $u$ and $d$ such that
for any non-proper HNN extension $G=K\rtimes_\alpha \z$ where $\alpha$
has no periodic conjugacy classes, we have $\omega(G)\geq k$.
\end{thm}
\begin{proof}
We use the proof of Theorem 5.3 and indicate where we branch off to deal
with the Klein bottle group. Let us take a subgroup 
$H_{i,j,k}=\langle t,x\rangle$ with $t\notin K$ but $x\in K-\{e\}$ as before
and again let $L$ (which we can assume is finitely generated)
be $H_{i,j,k}\cap K$. We still consider our ascending sequence of
subgroups $L_0=\langle x_0\rangle\cong\z$,
$L_1=\langle x_0,x_1\rangle,\ldots $ and everything will work as before
unless we come across a group $L_i$ on the way which is isomorphic to the
Klein bottle group. As this is our
only small subgroup which is non-abelian, we will have $L$ is itself
isomorphic to the Klein bottle group or we will drop out of the small 
subgroups during the sequence.

As the first Betti number of the Klein bottle group $L_i$ is 1, we have
$\beta_1(L)\leq 1$ by point (3) on the Alexander polynomial.
Now we use the local indicability of $K$ so we have $\beta_1(L)=1$
(as $L$ is not trivial). This means that the Alexander polynomial 
of the homomorphism from $H_{i,j,k}$ onto $\z$ with kernel $L$
can only be $t\pm 1$. Let us assume
it is $t-1$ as the other case just needs changes of signs. 
Therefore by point (1) there is (up to sign) only one homomorphism $\theta$
from $L$ onto $\z$ and the Alexander polynomial above  
tells us that $x_i=x_{i+1}$ holds in $L/L'$, so
$\theta(x_i)=1$ for all $i$. Now for any $i$ we have that $\theta$ 
restricts to a surjective homomorphism on $L_i$, so whenever
$L_i$ is isomorphic to the Klein bottle group we must have that $\theta$ 
agrees with the unique homomorphism (up to sign) from
the Klein bottle group onto $\z$. Any element in the Klein bottle group
can be expressed as $\alpha^i\beta^j$ for $i,j\in\z$ which will have image
$j$ under this homomorphism (changing $\beta$ to $\beta^{-1}$ if necessary).

We now show that if a Klein bottle group appears at all amongst the $L_i$,
it must first happen at $L_1$. Otherwise $L_1$ would be abelian and
if it is cyclic with
$x_0^ax_1^b=e$ for coprime $a$ and $b$ then, although this might not imply
that all subsequent $L_i$ are cyclic or abelian, it does mean that the
abelianisation of all subsequent $L_i$ will be $\z$ but the Klein bottle
group has abelianisation $C_2\times\z$. Thus $L_1$ is $\z\times\z$.
Suppose that everything from $L_1$ to
$L_{k-1}$ is $\z\times\z$ (so $tL_{k-1}t^{-1}$ is too)
but $L_k$ is the Klein bottle
group. Then $x_1,\ldots ,x_{k-1}$ will all be in the centre of $L_k$
because these elements commute with $x_0$ and $x_k$, as well as with each
other. But the centre of the Klein bottle group is generated by $\beta^2$,
thus $\theta(x_1)$ would be even, not 1, which is a contradiction if
$k\geq 2$.
Consequently our generators $x_0,x_1$ of the Klein bottle group $L_1$
must be $\alpha^i\beta$ and $\alpha^j\beta$ when written in this form. But this
implies that the relation $x_0^2=x_1^2(=\beta^2)$ holds in $L_1$ and hence
in $L$.

Let us now look at the case where $L_i$ is always a
Klein bottle group for $i\geq 1$ but where $L_1\neq L_2$.
Note that if a Klein bottle group is contained in
another then this has to be of finite index.
We can take our homomorphism $\theta$ from $L$ to
$\z$ and then to the cyclic group $C_2$, with the composition called
$\lambda$. Note that the kernel $M$
of this homomorphism will be isomorphic
to $\z\times\z$ as it contains the elements in $L$
of the form $\alpha^i\beta^j$ for even $j$. 
Also we have $t^{\pm 1}M t^{\mp 1}
=M$ because $M$ is the set of elements whose total exponent sum
is even in the $x_i$s. 
Consequently we can obtain an increasing sequence of subgroups
$M_i=M\cap L_i$ of $L$, each of which is 
of index 2 in the corresponding subgroup $L_i$ so that $L_1<L_2$ implies
$M_1<M_2$, and also that $M_1$ and $M_2$ are
isomorphic to $\z\times\z$. We have that $M_1=\langle z,y_1\rangle$
and $M_2=\langle z,y_1,y_2\rangle$, 
where we have set $z=x_0^2=x_1^2=x_2^2=\ldots$
and $y_i=x_{i-1}x_i$. Both $y_1$ and $y_2$ go to 2 under $\theta$, so on
writing the elements of $L_2$ in $\alpha,\beta$ form we have that
$y_1=\alpha^l\beta^2$ and $y_2=\alpha^m\beta^2$, with $l,m\neq 0$ 
as $z=\beta^2$ but $\z\not\cong M_1<M_2$ implies that $z,y_1,y_2$ are three
distinct elements.
As $\alpha$ and $\beta^2$ sit inside the torsion free
abelian group $M_2$, we cannot have a relation of the form 
$y_1^b=y_2^a$ because applying $\theta$ tells us that $a=b$
which would imply that $y_1=y_2$. Thus $\langle y_1,y_2\rangle$ is 
$\z\times\z$ and we will have a relation holding of the form $y_2^a=y_1^bz^c$ 
with $a\neq\pm 1$. Moreover we can assume that $a$ and $b$ are coprime, 
because applying $\theta$ tells us that $a=b+c$, so we could pull out
a common factor from all three indices. 

We now argue in a similar fashion as before by using the Alexander
polynomial but this time applied to the group $\langle t,y_i:i\in\z\rangle$
with homomorphism $t\mapsto 1,y_i\mapsto 0$, which works because 
$ty_it^{-1}=y_{i+1}$. Certainly the kernel 
$\langle y_i:i\in\z\rangle$ is finitely generated and has first Betti number
equal to 2 as it is not cyclic but sits inside $M\cong\z\times\z$. 
But on combining the relations $y_2^a=y_1^bz^c$ and $y_3^a=y_2^bz^c$
which hold in $M$, we obtain $y_3^a=y_2^{a+b}y_1^{-b}$. Thus we require
our monic quadratic Alexander polynomial to divide $at^2-(a+b)t+b$ which
is a contradiction. 

As for the case where $L_i$ stops being a Klein bottle group along the way,
we again try for an argument involving the centre of $L$. We have that
$L_0=\z$ with $L_1$ a Klein bottle group properly contained with finite
index in $L_2$ which we can assume is also a Klein bottle group or else
we have reached a growth constant bigger than 1.
We still have the homomorphism $\lambda$ from $L$ to $C_2$ but with a 
different kernel $M$. However if we
stabilise at $L_t=L_{t+1}$ then such a kernel will always be generated by 
$x_0^2,x_1^2,\ldots ,x_t^2$ (which here are all the same element $z$) and
$x_0x_1,x_1x_2,\ldots ,x_{t-1}x_t$, which here are called 
$y_1,y_2,\ldots ,y_t$. Now if we restrict $\lambda$ to
$\langle x_i,x_{i+1}\rangle$ which being a conjugate of $L_1$
is also a Klein bottle group, we have as before 
that the intersection of this subgroup with $M$ is $\z\times\z$ and
generated by $x_i^2=x_{i+1}^2=z$ and $x_ix_{i+1}=y_{i+1}$, 
so for any $i$ the two elements $z$ and $y_i$ commute.
This means that $z$ is in the centre of 
$M$. As for $y_{i+1}$, we have that 
$L_1\cap M=\langle z,y_1\rangle$ has finite
index in $L_2\cap M=\langle z,y_1,y_2\rangle$ which is also
isomorphic to $\z\times\z$, because there is only one homomorphism 
from a Klein bottle group to $C_2$ which factors through $\z$.
Thus we must have a relation of the form $y_2^a=y_1^bz^c$ as before,
where we can again assume that $a$ and $b$ are coprime. 
But now we argue as previously that 
$y_1^{b^2}=(y_2^b)^az^{-cb}=(y_3^az^{-c})^az^{-cb}$ thus
$y_1$ to the power $b^2$ commutes with $y_3$ and so on.
We then get that $y_i$ to a high power of both $b$ and (by working
back down again) $a$ is in the
centre, thus $y_i$ is too and so $M$ is abelian after all. This tells us
that $L$ has an abelian subgroup of index 2, so it can only be abelian or
the Klein bottle group.

Thus we are done unless $L_1=L_2$ is the Klein bottle group. We now
argue the other way with $L'_0=\langle x_1\rangle$, 
$L'_1=\langle L'_0,x_0\rangle=L_1$, 
$L'_2=\langle L'_1,x_{-1}\rangle$ and we are done
unless $L'_2=L'_1$, in which case $L=L_1$ with
$H_{i,j,k}$ equal to $L\rtimes_\beta\z$ for some automorphism $\beta$.
But a Klein bottle group has a
characteristic subgroup $C\cong\z\times\z$ with quotient $C_2\times C_2$
so $H_{i,j,k}$ has a subgroup $(\z\times\z)
\rtimes_\beta\z$ of index 4. This gives us a lower bound 
for the growth constant of $H_{i,j,k}$ unless the subgroup is
virtually nilpotent, in which case as before $\beta$ and hence $\alpha$
has a periodic conjugacy class.
\end{proof}

As before we have a Corollary to this result which is proved similarly.
\begin{co}
Let $K$ be a finitely generated locally indicable group 
where there exists $u>1$ and $d\in\n$ such that if $H$ is a 
finitely generated 
subgroup of $K$ then either $\omega(H)\geq u$ or $H\cong\z^n$ for $n\leq d$
or $H$ is isomorphic to the Klein bottle group.
Then there exists $k>1$ depending only on $u$ and $d$ such that
for any non-proper HNN extension $G=K\rtimes_\alpha \z$ where $\alpha$
has no periodic conjugacy classes and for any finitely generated subgroup
$S$ of $G$, we have that either $\omega(S)\geq k$ or $S$ is $\z^n$ for
$n\leq d$ or is the Klein bottle group.
\end{co}
\begin{proof}
Once again, Theorem 6.1 applies to $S$ unless $R=S\cap K$ is $\z^n$ or the
Klein bottle group. But then $S$ is a non-proper extension of one of these
groups, which will either mean that $S$ has uniform exponential growth
depending only on $d$, or we have a periodic conjugacy class as at the end
of the proof of Corollary 5.4.
\end{proof}
As for applications, we have some when $d\leq 2$.
\begin{co}
Let $n\in\n$. Then there exists $k>1$ depending only on $n$ with the
following property.
Let $K$ be any finitely generated coherent locally indicable group of
cohomological dimension 2 and suppose
we form a iterated sequence of $n$ non-proper
HNN extensions to obtain the group $G$ where every automorphism
has no periodic conjugacy classes. Then for any finitely generated subgroup
$S$ of $G$, we have that either $\omega(G)\geq k$ or $S$ is $\z^n$ for
$n\leq 2$ or is the Klein bottle group.
\end{co}
\begin{proof}
First let us consider $G_1=K\rtimes_{\alpha_1}\z$ where $\alpha_1$ has no
periodic conjugacy classes. The cohomological dimension means $K$ is
torsion free. On taking a finitely generated (hence presented) subgroup
$H$ of $K$ and using local indicability, we have a kernel of a homomorphism
from $H$ to $\z$ which if infinitely generated gives us 
$\omega(H)\geq 2^{1/4}$, and which if finitely generated must be free, as
mentioned in Proposition 2.3, giving us a similar bound 
unless $H$ is $\z,\z\times\z$ or the Klein bottle. 
Thus we can apply Corollary 6.2 with 
$u=2^{1/4}$ and $d=2$ to obtain our conclusion for $G_1$. Now we form
$G_2=G_1\rtimes_{\alpha_2}\z$. Although we may have lost coherence and have
probably increased the cohomological dimension by 1, $G_1$ is still
locally indicable because a finitely generated subgroup $S$ of $G_1$ will
either lie in $K$ or will have a non-trivial image in $\z$ on restriction
of the associated homomorphism of the HNN extension. Thus we are now able
to apply the Corollary repeatedly.
\end{proof}

The conditions on $K$ in Corollary 6.3 may look restrictive but we give
some cases where they are satisfied.\\
(1) If $K$ is a free-by-cyclic group of the form $F_n\rtimes_\alpha\z$,
with coherence established in \cite{fh}.
Thus we get uniform uniform growth for
iterated sequences of non-proper HNN extensions of finitely generated
free groups, as long as all automorphisms except for the first 
have no periodic conjugacy classes. We also have the same result
if all automorphisms except for the last have no periodic conjugacy classes,
by Proposition 5.1 and Corollary 5.2.\\
(2) If $K$ is a torsion free 1-relator group which is coherent, as 
mentioned in Proposition 2.3. Thus we can take such a group and apply
repeated non-proper HNN extensions, retaining uniform uniform
exponential growth if the automorphisms are non-periodic.\\ 
(3) Generalised Baumslag-Solitar groups, as mentioned at the end of
Section 4. There we got uniform uniform growth for non-proper HNN
extensions of generalised
Baumslag-Solitar groups which do not contain $\z\times\z$ formed
by any
automorphism. Now we get uniform uniform growth for non-proper HNN
extensions of any generalised Baumslag-Solitar group, but using non-periodic
automorphisms and repeated extensions thereof.\\
(4) We can also apply this result to finitely generated
linear locally indicable groups $K$ of cohomological dimension 2, in which case
we can lose the coherence condition. This is because by \cite{breu} we have
$u>1$ depending only on the degree $r$ of $K$ such that any finitely
generated subgroup $H$ of $K$ has $\omega(H)\geq u$ or is virtually
soluble. But the finitely generated soluble groups of cohomological
dimension at most 2 are known by \cite{gild} to be $\{e\},\z,\z\times\z$, the
Klein bottle or the Baumslag Solitar groups $BS(1,m)$ for 
$m\neq 0,\pm 1$ with the latter family
having growth constant at least $2^{1/4}$. Moreover this result also
applies to virtually soluble groups by Corollaries 3(ii) and 2 of
\cite{krop} so we can apply Corollary 6.2 $n$ times with $d=2$
to obtain $k$, which depends only on the degree $r$ and $n$, such that any
iterated sequence of non-proper HNN extensions of $K$ using automorphisms
with no periodic conjugacy classes results in a group $G$ where all
finitely generated subgroups either have growth constant at least $k$ or
must be $e,\z,\z\times\z$ or the Klein bottle group.

Finally we finish with a comment on amalgamated free products. In 
\cite{bdlh} (non-trivial)
amalgamated free products are also considered, where it is 
again shown that such a group has growth constant at least $2^{1/4}$
unless the amalgamated subgroup has index 2 in both factors. It is remarked
in the latter case that the resulting group may or may not have uniform
exponential growth but it may be fruitful to regard it as an HNN extension.
It is well known that a group $F$ decomposes as an amalgamated
free product $G*_{A=B}H$ with $A$ of index 2 in $G$ and the isomorphic 
subgroup $B$ of index 2 in $H$ if and only if there is a surjection from
$F$ to $C_2*C_2$. Consequently $F$ has an index 2 subgroup $S$ surjecting onto
$\z$, giving $\omega(S)\geq 2^{1/4}$ and thus
$\omega(F)\geq 2^{1/12}$ if $S$
is a proper HNN extension. If $S$ is a non-proper HNN extension it may not
have uniform exponential growth anyway, but we can at least
see if any of the results here apply to $S$.

\end{document}